\documentclass[a4paper, 11pt, oneside]{scrartcl}
  \usepackage[english]{babel}
  \usepackage[utf8]{inputenc}
  \usepackage{paralist}
  \usepackage{mathtools}
  \usepackage{amssymb}
  \usepackage{microtype}
  \usepackage{enumerate}
  \usepackage{mathrsfs}
  \usepackage{xcolor}
  \usepackage{graphicx}
  \usepackage{paralist}
  \usepackage[normalem]{ulem}
  \usepackage{tikz}
  \usetikzlibrary{matrix,arrows,positioning,shapes.geometric}
  \usepackage[square,numbers]{natbib}
	\usepackage{hyperref}
	\usepackage{amsthm}
	\usepackage[capitalise]{cleveref}
	\usepackage{authblk}
	
  \theoremstyle{plain}
  \newtheorem{theorem}{Theorem}[section]
  
  \newtheorem{prop}[theorem]{Proposition}
  \newtheorem{corollary}[theorem]{Corollary}

  \theoremstyle{definition}
  \newtheorem{defn}[theorem]{Definition}
  
  \theoremstyle{remark}
  \newtheorem{remark}[theorem]{Remark}

  \usepackage[left=3cm,right=3cm,top=3.5cm,bottom=4.2cm]{geometry}
  
  \DeclarePairedDelimiter{\abs}{\lvert}{\rvert}

  \DeclarePairedDelimiter{\fp}{\{}{\,|\kern-0.2em \}}
  \DeclarePairedDelimiter{\FK}{[}{\,|\kern-0.2em ]}
  
  \DeclarePairedDelimiter{\bb}{\{}{\}_{\mathrm{GC}}}

  \newcommand{\K}{\mathbb{K}}

  \newcommand{\Z}{\mathbb{Z}}

  \renewcommand{\d}{\mathrm{d}}

  \renewcommand{\phi}{\varphi}

  \newcommand{\sign}{\operatorname{sign}}

  \newcommand{\pr}{\operatorname{pr}}

  \newcommand{\img}{\operatorname{Im}}
  \newcommand{\id}{\operatorname{id}}

  \newcommand{\del}{\partial}
  \renewcommand{\O}{\mathcal{O}}
  \newcommand{\factor}[2]{\left.\raisebox{.2em}{$#1$}\middle/\raisebox{-.2em}{$#2$}\right.}
  
  \newcommand{\ph}{[[t]]}

  \newcommand{\g}{\mathfrak{g}}

  \newcommand{\ad}{\operatorname{ad}}

  \newcommand{\Hom}{\operatorname{Hom}}

  \newcommand{\Der}{\operatorname{Der}}

	\newcommand{\abLC}{\mathfrak{\widetilde C}}
  
	\newcommand{\bLC}{\mathfrak{C}}
  	
	\newcommand{\LC}{\operatorname{C}_{\mathrm{CE}}}
  \newcommand{\LH}{\operatorname{H}_{\mathrm{CE}}}
  \newcommand{\aLC}{\operatorname{\widetilde{C}}_{\mathrm{CE}}}
	\newcommand{\aLH}{\operatorname{\widetilde{H}}_{\mathrm{CE}}}
	\newcommand{\aC}{\operatorname{C}}
	
	\newcommand{\eC}{\mathcal{C}}
  \newcommand{\delaa}{\del_{\alpha\alpha}}
  \newcommand{\delma}{\del_{\nu\alpha}}
  \newcommand{\delam}{\del_{\alpha\nu}}
  \newcommand{\delmm}{\del_{\nu\nu}}

  \newcommand{\aHC}{\operatorname{\tilde{H}\!C}}

	\DeclareMathOperator*{\cycl}{\scalebox{1.9}{\raisebox{-0.4ex}{$\circlearrowleft$}}}

  \newcommand{\vera}[1]{%
  \begin{tikzpicture}[scale=0.3,point/.style={draw,shape=circle,fill=blue,minimum size=2,inner sep=0}, circ/.style={draw,shape=circle,minimum size=5,inner sep=0}]
   \node [circ] (a0) at ( #1/2 + 0.5,1) {};
   \draw (a0) -- +(0,0.7) ;
   \foreach  \i in{1,...,#1}
   {
  \draw (\i,0) -- (a0);
  }
  \end{tikzpicture} 
  }
  
  \newcommand{\verm}[1]{%
  \begin{tikzpicture}[scale=0.3,point/.style={draw,shape=circle,fill=blue,minimum size=2,inner sep=0},  circ/.style={draw,shape=circle,minimum size=5,inner sep=0}]
   \node [point] (a0) at ( #1/2 + 0.5,1) {};
   \draw (a0) -- +(0,0.7) ;
   \foreach  \i in{1,...,#1}
   {
  \draw (\i,0) -- (a0);
  }
  \end{tikzpicture} 
  }
  
  \newcommand{\vermk}{%
  \begin{tikzpicture}[scale=0.3,point/.style={draw,shape=circle,fill=blue,minimum size=2,inner sep=0},  circ/.style={draw,shape=circle,minimum size=5,inner sep=0}]
   \node [point] (a0) at ( 2,1.3) {};
   \draw (a0) -- +(0,0.7) ;
   \draw (1,0) -- (a0);
   \draw (3,0) -- (a0);
   \draw[dotted] (1,0) -- node[above,scale=.7] {$k$} (3,0);
  \end{tikzpicture} 
  }
  
  \newcommand{\verak}{%
  \begin{tikzpicture}[scale=0.3,point/.style={draw,shape=circle,fill=blue,minimum size=2,inner sep=0},  circ/.style={draw,shape=circle,minimum size=5,inner sep=0}]
   \node [circ] (a0) at ( 2,1.3) {};
   \draw (a0) -- +(0,0.7) ;
   \draw (1,0) -- (a0);
   \draw (3,0) -- (a0);
   \draw[dotted] (1,0) -- node[above,scale=.7] {$k$} (3,0);
  \end{tikzpicture} 
  }

  \allowdisplaybreaks

  \title{$\alpha$-type Chevalley-Eilenberg cohomology of Hom-Lie algebras and bialgebras}
	
	 \author[1]{Benedikt Hurle}
\author[1]{Abdenacer Makhlouf}
\affil[1]{IRIMAS, Département de Mathématiques, Universit\'e de Haute Alsace, Mulhouse (France)}

\begin{document}
  
\maketitle

\begin{abstract}
	The purpose of this paper is to define an $\alpha$-type  cohomology, which we call $\alpha$-type Chevalley-Eilenberg cohomology, for Hom-Lie algebras. We relate it to the known Chevalley-Eilenberg cohomology and provide explicit computations for some examples. Moreover, using this cohomology we study formal deformations of Hom-Lie algebras, where the bracket as well as the structure map $\alpha$ are deformed. Furthermore, we provide a generalization of the Grand Crochet and study, in a particular case, the $\alpha$-type  cohomology for Hom-Lie bialgebras.
\end{abstract}


\section*{Introduction}

Hom-Lie algebras were introduced in \citep{silvestrov06} by considering deformations of Lie algebras of vector fields by $\sigma$-derivations. They were further studied in \citep{homalg}, where Hom-associative algebras were defined and the definition modified slightly. The main feature is that the identities are modified using a homomorphism denoted usually by  $\alpha$.   

In this paper we define an $\alpha$-type Chevalley-Eilenberg cohomology for Hom-Lie algebras. It is an extension of the cohomology for Hom-Lie algebras considered in \citep{homcoho}. It is build similar to the $\alpha$-type Hochschild cohomology for Hom-associative algebras  defined in \citep{homhcoho}. It allows us to study deformations of Hom-Lie algebras, where the bracket and the structure map $\alpha$ are deformed. We obtain the expected  results, that is the first order term of a deformation is a 2-cocycle and more generally the order by order construction of deformations is equivalent to solving  equations in the third cohomology space. To prove this we define an $L_\infty$ algebra, with Hom-Lie algebras as Maurer-Cartan elements and related to the $\alpha$-type cohomology, using a graph complex corresponding to a free operad. Notice that deformations of Hom-Lie algebras have also been considered in \citep{makhlouf_ParamFormDefHomAssLie} and \citep{homcoho}, but where only the bracket is deformed. 
In the case the Hom-Lie algebra is just an ordinary Lie algebra, i.e.\ the structure map is the identity, we compute the $\alpha$-type cohomology in terms of the ordinary one. More generally  for Hom-Lie algebras of Lie type, which includes the ones with invertible structure map, we compute the cohomology and relate it to the cohomology of a Lie algebra with an endomorphism.   Moreover, concrete examples are provided, we compute the $\alpha$-type cohomology explicitly for some low-dimensional Hom-Lie algebras with non-invertible structure map. 
We also prove a generalization of the well-known Whitehead theorem for simple Lie algebras to the Hom-case.  

Furthermore, we  give a generalization to the Hom-case of the grand crochet, which can be used to study Lie bialgebras. Hom-Lie bialgebras have been studied in different approaches \citep{MR3310684,sheng1}. The super case has been discussed in \citep{fadous}.  In the associative case a cohomology for Hom-bialgebras has been defined  in \citep{gshom}. This allows us to write down in an easy way the cohomology for Hom-Lie bialgebras with fixed structure map for an example studied in a less general case in \citep{shengbb}.

This paper is structured as follows: in the first section we recall the basic definitions of Hom-Lie and associative algebras and their modules. In  \cref{sc:coho} we define the $\alpha$-type Chevalley-Eilenberg cohomology and gives some of its properties. In \cref{sc:linfty} we give an $L_\infty$ structure such that the Maurer-Cartan elements are Hom-Lie algebras and which can be used to study  deformations of Hom-Lie algebras using the $\alpha$-type cohomology.
In the next section we compute the $\alpha$-type cohomology, or at least its dimensions, for some concrete examples. 
In \cref{sc:grandcrochet} we give a generalization of the grand crochet or big bracket constructed in \citep{lecomte,kosmann} and for the Hom-case in \citep{shengbb}. Using this we define in \cref{sc:bialgcoho} a cohomology for Hom-Lie bialgebras with fixed structure map. Finally we give in the last section some remarks on how this can be generalized to an $\alpha$-type cohomology for Hom-Lie bialgebras. 

\section{Basics}
Let $\K$ be a field of characteristic zero, but note that most constructions should also work in other characteristics (not 2) or  if $\K$ is a ring  containing the rational numbers. 
  
We recall  basic definitions and constructions of Hom-Lie algebras and Hom-associative algebras.

	A Hom-module $(V,\alpha)$ is a vector space $V$ together with a linear map $\alpha$, called the structure map. 
	A morphism between Hom-modules $(V,\alpha)$ and $(W, \beta)$ is a linear map $\phi: V \to W$ such that $\phi \circ \alpha = \beta \circ \phi$.

\begin{defn}[Hom-Lie algebra]
	A Hom-Lie algebra $(\g,[\cdot,\cdot],\alpha)$ is a Hom-module $(\g,\alpha)$ with a skew-symmetric linear map $[\cdot,\cdot]:  \g \otimes \g \to \g$, called Hom-Lie bracket, such that the Hom-Jacobi identity is satisfied and the bracket is compatible with $\alpha$, i.e.\ for $x,y,z \in \g$
	\begin{gather}
		[[x,y],\alpha(z)] +  [[y,z],\alpha(x)] +[[z,x],\alpha(y)]  =0 ,              \\
		[\alpha(x),\alpha(y)]    = \alpha([x,y]). 
	\end{gather}
	The second equation is called multiplicativity and is not required in some papers.
	We will often write $(\g,\nu,\alpha)$ instead of $(\g,[\cdot,\cdot],\alpha)$, where  $\nu: \g \otimes \g \to \g$ is given by $\nu(a \otimes b) =[a,b]$.
\end{defn}

A morphism of Hom-Lie algebras is a morphism of the underlying Hom-modules, which also preserves the product.  Morphisms for  other types of  Hom-algebras and Hom-coalgebras are defined similarly.

A well known method introduced by D. Yau for constructing Hom-Lie algebras is given as follows:

\begin{prop}
	Let $(\g,\nu,\alpha)$ be a Hom-Lie algebra and $\gamma:\g \to \g$ be a Hom-Lie algebra morphism, then $(\g, \gamma \circ \nu , \gamma \circ \alpha)$ is  again a Hom-Lie algebra, which we denote by $\g_\gamma$ and call it   the Yau twist of $\g$ by $\gamma$.
\end{prop}

If a Hom-Lie algebra is the Yau twist of an ordinary Lie algebra, we say  it is of Lie-type.

\begin{defn}[Hom-associative algebra]
	A Hom-associative algebra $(A,\mu,\alpha)$  consists of   a Hom-module $(A,\alpha)$ and a linear map $\mu:A \otimes A \to A$, such that 
	\begin{align}
		\mu \circ  ( \id \otimes \alpha) =  \mu  \circ  ( \alpha \otimes \id) \\	
		\mu \circ  (\alpha \otimes \alpha) = \alpha \circ \mu.
	\end{align}
	\end{defn}
  
\begin{prop}\citep{homalg}
	Let $(A,\mu,\alpha)$ be a Hom-associative algebra, then the commutator $[x,y] = xy -yx$ defines a Hom-Lie algebra structure on $A$, which we denote by $A_L$.
\end{prop}

Next, we give the definition of modules for Hom-Lie and Hom-associative algebras.

\begin{defn}[Representation of a Hom-Lie algebra]
	Let $(\g,[\cdot,\cdot],\alpha)$ be a Hom-Lie algebra.
	A  representation of $\g$ is a Hom-module $(V,\beta)$ with an action $\rho:\g \otimes V \to V, g \otimes v \mapsto g
	 \cdot v$, such that for all  $x,y \in g$ and $v \in V$
	\begin{equation}
		[x,y] \cdot \beta(v) = \alpha(x) \cdot ( y \cdot v) -  \alpha(y) \cdot ( x \cdot v).
	\end{equation}
	We require that it is multiplicative, i.e. 
	\begin{equation}
		\beta(x \cdot  v) = \alpha(x) \cdot \beta(v).
	\end{equation}
	We call $(M,\rho,\beta)$ also a Hom-Lie module over $\g$ or simply a $\g$-module.
\end{defn}
    
It is clear that a Hom-Lie algebra $\g$ is a $\g$-module  by the adjoint action, given by $\ad_x y := x \cdot y = [x,y]$.

\begin{defn}
	Let $(A,\mu,\alpha)$ be a Hom-associative algebra and $(M,\beta)$ be a Hom-module. Further let $\rho:A \otimes M \to M, (a \otimes m)\mapsto a \cdot m$ be a linear map, then $(M,\rho)$ is called an (left) $A$-module  if 
	\begin{gather}
		(a b) \cdot \beta(m)  = \alpha( a ) ( b \cdot m), \\	
		\beta(a \cdot m) = \alpha(a) \cdot \beta(m).
	\end{gather}
Similarly one can define  right $A$-modules.

	An $A$-bimodule is a Hom-module $(M,\beta)$, with two maps $\rho: A \otimes M \to M, a \otimes m \mapsto a \cdot m$ and $\lambda: M \otimes A \to M, a \otimes m \mapsto m \cdot a$, such that $\rho$ is a left and $\lambda$ a right module structure and 
	\begin{equation}
		\alpha(a) \cdot( m \cdot b) = ( a \cdot m) \cdot \alpha(b).
	\end{equation}
\end{defn}

Note that a left $A$-module can be considered as an $A$-bimodule, where the right action is trivial, i.e.\ identically zero.

Obviously $A$ is an $A$-(bi)module, where the action is giving by the multiplication in $A$.

\begin{prop}
	Let $(A,\mu,\alpha)$ be a Hom-associative algebra and $(M,\beta)$ an $A$-bimodule. Then $M$ is a representation of the Hom-Lie algebra $A_L$, with action $a \cdot_L m =  a \cdot m - m \cdot a$, which we denote by $M_L$.
	We will simply write $a \cdot m$ for $a \cdot_L m$  if it is clear that we consider the Hom-Lie action.   
\end{prop}
\begin{proof}
	 One has to verify 
	\begin{equation}
		[a ,b ] \cdot_L \alpha(m) = \alpha(a) \cdot_L(b \cdot_L m) - \alpha(b) \cdot_L(a \cdot_L m).
	\end{equation}
\end{proof}

Next we define the dual notion to Hom-Lie algebra.

For a vector space $\g$ and $x,y,z \in \g$, we denote by $\sigma$ the cyclic permutation map $ x \otimes y \otimes z \mapsto z \otimes x \otimes y$ and by $\tau: \g \otimes \g \to \g \otimes \g$ the flip, i.e. $\tau(x \otimes y) = y \otimes x$. We call a linear map $\delta :\g \to \g \otimes \g$ skewsymmetric if $\delta = - \tau\circ \delta$.

\begin{defn}
	A Hom-Lie coalgebra $(\g,\delta,\alpha)$ is a Hom-module $(\g,\alpha)$ with a skew-symmetric cobracket $\delta:\g \to \g  \otimes \g$, satisfying.
	\begin{equation}
		\delta \circ  ( \delta \otimes \alpha)\circ (\id + \sigma + \sigma^2)  =0.
	\end{equation}
	We also require it to be multiplicative, i.e. $\delta \circ  \alpha = (\alpha \otimes \alpha) \circ \delta$.
\end{defn}

Let $(\g,\nu,\alpha)$ be a Hom-Lie algebra, then a derivation is a linear map $\phi : \g \to \g$, such that $\phi([x,y])= [\phi(x),y] + [x,\phi(y)]$ and an $\alpha$-derivation is a linear map $\psi:\g \to \g$, such that $\psi([x,y]) =[\psi(x),\alpha(y)]+ [\alpha(x),\psi(y)]$. We denote the set of derivations of $\g$ by $\Der(\g)$ and the set of $\alpha$-derivations by $\operatorname{\alpha-Der}(\g)$.

We  recall the definition of a Hom-Lie bialgebra.

\begin{defn}[Hom-Lie bialgebra]
	A Hom-Lie bialgebra is a tuple $(\g,\nu,\delta,\alpha,\beta)$, such that 
	$(\g,\nu,\alpha)$ is a Hom-Lie algebra, 
	$(\g,\delta,\beta)$ is a Hom-Lie coalgebra and they are compatible in the sense that
	\begin{equation}\label{eq:liebi}
		\begin{split}
			\delta([x,y]) = \alpha(x^{(1)}) \otimes [x^{(2)} , \beta(y)] + [x^{(1)}, \beta(y)] \otimes \alpha(x^{(1)}) \\
			+  [ \beta(x),y^{(1)}] \otimes \alpha(y^{(1)}) + \alpha(y^{(1)}) \otimes [\beta(x), y^{(2)}].
		\end{split}
	\end{equation}
	Here we use Sweedler's  notation $\delta(x) = x^{(1)} \otimes x^{(2)}$ for the cobracket. Note that on the right hand side there is an implicit sum.
	We also require that $\beta\circ  \nu = \nu\circ  (\beta \otimes \beta)$, $\delta\circ  \alpha = (\alpha \otimes \alpha)\circ \delta$ and $\beta\circ \alpha = \alpha\circ \beta$.
\end{defn}

The condition~\eqref{eq:liebi} is often written as
\begin{equation}
	\delta([x,y]) = \ad_{\beta(x)} \delta(y) - \ad_{\beta(y)} \delta(x), 
\end{equation}
where the adjoint representation of $\g$ on $\g^{\otimes k}$ is defined by
\begin{equation}
	\ad_x (y_1 \otimes \dots \otimes y_k) = \sum_{i=1}^k \alpha(y_1) \otimes \dots \otimes [x,y_i] \otimes  \dots \alpha(y_k).
\end{equation}
Often only the cases $\alpha = \beta$ or $\alpha = \beta^{-1}$ are considered.

A Hom-Lie bialgebra morphism is a  morphism of the underlying Hom-Lie  algebra and Hom-Lie coalgebra.

\begin{prop}
	Let $(\g,\nu,\delta,\alpha,\beta)$ be a Hom-Lie bialgebra and $\gamma:\g \to \g$ be a bialgebra morphism then $(\g,\gamma \circ \nu,\delta,\gamma \circ\alpha,\beta)$,  $(\g,\nu,\delta \circ\gamma,\alpha,\gamma \circ\beta)$  and  $(\g,\gamma \nu,\delta\circ \gamma,\gamma\circ \alpha,\gamma \circ\beta)$ are again  Hom-Lie bialgebras. 
\end{prop}
\begin{proof}
	Straightforward calculation.
\end{proof}

\section{\texorpdfstring{$\alpha$}{alpha}-type cohomology for  Hom-Lie algebras}\label{sc:coho}

Similarly to \citep{homhcoho}, we define a cohomology for Hom-Lie algebras. Let $(\g,\nu,\alpha)$ be  a	Hom-Lie algebra and $(M,\beta)$ be  a $\g$-module. We denote by $\Lambda^k \g$ the $k$-th exterior power of $\g$.
Then the  complex for the cohomology of $\g$ with values in $M$ is given by 
\begin{equation}
	\aLC^n(\g,M) = \aC^n_\nu(\g,M) \oplus \aC^n_\alpha(\g,M) = \Hom(\Lambda^{n} \g,M) \oplus \Hom(\Lambda^ {n-1} \g,M). 
\end{equation}
Here $\Hom(\Lambda^0 \g,M)$ is set to be $\{0\}$, instead of $\K$ as usual, since otherwise $\alpha^{-1}$ would be needed in the definition of the differential.
We write $(\phi, \psi)$ or $\phi + \psi$ with $\phi \in \aC^\bullet_\nu(\g,M)$ and $\psi \in \aC^\bullet_\alpha(\g,M)$  for an element in $\aLC^\bullet(\g,M)$.

We define four maps, with domain and range given in the following diagram:
\begin{center}
	\begin{tikzpicture}
		\matrix (m) [matrix of math nodes,row sep=3em,column sep=5em,minimum width=2em] {
			\aC^n_\nu & \aC^{n+1}_\nu  \\
			\aC^n_\alpha & \aC^{n+1}_\alpha \\
		};  
		\path[->,auto] (m-2-1) edge node[swap]{$\delaa$}                  (m-2-2);
		\path[->,above] (m-2-1) edge node[below] {$\delam$}                  (m-1-2);
		\path[->,below] (m-1-1) edge node[above] {$\delma$}                  (m-2-2);
		\path[->,auto] (m-1-1) edge node {$\delmm$}                  (m-1-2);
		\path   (m-2-1) edge[draw=none]   node [sloped] {$\oplus$} (m-1-1);
		\path   (m-2-2) edge[draw=none]   node [sloped] {$\oplus$} (m-1-2);
	\end{tikzpicture}
\end{center}
\begin{align}
	(\delmm \phi)(x_1,\dots,x_{n+1}) & = \sum_{i=1}^{n+1}  (-1)^{i+1}\alpha^{n-1}(x_i) \cdot \phi(x_1,\dots,x_{n+1})                                     \\
	                                 & - \sum_{i<j}(-1)^{i+j-1} \phi([x_i,x_j],\alpha(x_1 ),\dots,\hat x_i,\hat x_j,\dots,\alpha(x_{n+1}))    \nonumber                   \\
	(\delaa \psi)(x_1,\dots,x_{n})   & = \sum_{i=1}^{n}  (-1)^{i+1} \alpha^{n-1}(x_i) \cdot \psi(x_1,\dots,x_{n})                                        \\
	                                 & - \sum_{i<j}(-1)^{i+j-1} \psi([x_i,x_j],\alpha(x_1 ),\dots,\hat x_i,\hat x_j,\dots,\alpha(x_n))       \nonumber                \\
	(\delma \phi) (x_1,\ldots, x_n)    & = \beta( \phi(x_1,\ldots, x_n) ) - \phi(\alpha(x_1),\ldots,\alpha( x_n))                                          \\
	(\delam \psi)(x_1,\ldots, x_{n+1}) & = \sum_{i\leq j} (-1)^{i+j-1} [ \alpha^{n-2}(x_i),\alpha^{n-2}(x_j)] \cdot \psi(x_1,\dots,\hat x_i,\hat x_j,x_{n+1}), 
\end{align}
where $x_1,\dots, x_{n+1} \in \g$. 

The sign  given by $(-1)^\cdot$ is always determined by the permutation of the $x_i$.

We have the following main theorem.
\begin{theorem}
	Let $(\g,\nu,\alpha)$ be a Hom-Lie algebra and $(M,\beta)$ be a $\g$-module. 
	Further let 
$\del: \aLC^n(\g,M) \to \aLC^{n+1}(\g,M)$ be a map defined by $\del(\phi,\psi) = (\delmm \phi -  \delam \psi , \delma \phi - \delaa \psi)$. Then the pair $(\aLC^\bullet(\g,M),\del)$ is a cohomology complex.
\end{theorem}
\begin{proof}
	This is a straightforward calculation. One has to take care of the signs, but as stated above most of the sign come from  permutations of the $x_i$. We will omit this in the following, and simply write $\pm$, since this simplifies the formulas and the correct sign is easy to obtain. Further in the sum over $i,j,\dots$ the elements corresponding to $x_i,x_j,\dots$  are omitted from the sequence $x_1, \dots, x_n$, such that each $x_i$ appears once in each expression.
	\begin{align}
		\delmm (\delmm \phi) (x_1,\dots, x_n) &= \sum_{i=1}^{n-1} \sum_{j=1, i\neq j}^{n-1}   \pm \alpha^{n}(x_i) \cdot (\alpha^{n-1} (x_j) \cdot \phi(x_1, \dots , x_n ) ) \label{eq:lico1}\\
		&- \sum_{i} \sum_{j<k, j,k \neq i} \pm \alpha^n (x_i) \cdot \phi([x_j,x_k], \alpha(x_1), \dots, \alpha(x_n))\label{eq:lico2} \\
		&-  \sum_{i<j} \sum_{k\neq i,j} \pm \alpha^n(x_k) \cdot \phi([x_i,x_j], \alpha(x_1), \dots, \alpha(x_n))\label{eq:lico3} \\
		&- \sum_{i<j}\pm \alpha^{n-1} ([x_i,x_j]) \cdot \phi(\alpha(x_1) ,\dots ,\alpha(x_n) ) \label{eq:lico4}\\
		&+  \sum_{i<j} \sum_{k<l,k,l \neq i,j} \pm \phi(\alpha[x_k,x_l],\alpha([x_i,x_j]), \alpha^2(x_1), \dots, \alpha^2(x_n) ) \label{eq:lico5}\\
		&+ \sum_{i<j} \sum_{k\neq i,j} \pm  \phi([[x_i,x_j],\alpha(x_k)], \alpha^2(x_1), \dots, \alpha^2(x_n)) \label{eq:lico6}\\
		&= \sum_{i<j} \sum_{i<j} \pm \alpha^{n-1} ([x_i,x_j]) \cdot \phi(\alpha(x_1) ,\dots ,\alpha(x_n) ) \label{eq:lico7} \\
		&-  \sum_{i<j} \pm \alpha^{n-1} ([x_i,x_j]) \cdot \beta \phi(x_1 ,\dots , x_n) \label{eq:lico8}.
	\end{align}		
	The terms~\eqref{eq:lico2} and~\eqref{eq:lico3}  cancel each other and~\eqref{eq:lico5} cancel itself due to antisymmetry, and~\eqref{eq:lico6}  due to the Hom-Jacobi identity.  \Cref{eq:lico1} gives the term~\eqref{eq:lico8} by the Hom-Jacobi identity and~\eqref{eq:lico4} is equal to~\eqref{eq:lico7}.
It is easy to see that this is the same as $\delam \delma \phi$. 

One can compute $\delaa( \delaa \psi)$ similar to 	$\delmm (\delmm \phi)$  and gets 
\begin{align*}
	 \sum_{i<j} \sum_{i<j} \alpha^{n} ([x_i,x_j]) \cdot \phi(\alpha(x_1) ,\dots )  -  \sum_{i<j} \alpha^{n} ([x_i,x_j]) \cdot \beta \phi(x_1 ,\dots ).
\end{align*}
This is the same as $\delma \delam \psi$.
It remains to show that  $\delma \delmm = \delaa \delma$ and $\delam \delaa = \delmm \delma$.  
We compute 
\begin{align*}
	\delma \delmm \phi (x_1,\dots, x_{n+1}) &= \sum_{i=1}^n (-1)^{i-1}\alpha^n(x_i) \cdot \beta \phi (x_1, \dots,x_{n+1})\\
	&- \sum_{i=1}^n (-1)^{i-1} \alpha^n (x_i) \cdot \phi(\alpha(x_1), \dots,\alpha(x_{n+1})) \\
	& - \sum_{i<j} (-1)^{i+j-1} \phi(\alpha([x_i,x_j]),\alpha^2(x_1), \dots ,\alpha^2(x_{n+1}) )  \\
	&+  \sum_{i<j} (-1)^{i+j-1}  \beta\phi([x_i,x_j],\alpha(x_1), \dots, \alpha(x_{n+1})) \\
	&= \delaa \delma \phi (x_1,\dots, x_{n+1})
\end{align*}
and 
\begin{align*}
	\delam \delaa \psi(x_1,\dots, x_{n}) &= \sum_{i<j}\sum_{k \neq i,j} \pm \alpha^{n-1}([x_i,x_j] \cdot \alpha^{n-1}([x_i,x_j],)) \cdot \psi(x_1, \dots, x_n) \\
	&-  \sum_{i<j} \sum_{k<l,k,l\neq i,j} \pm\alpha^{n-1}([x_i,x_j]) \cdot \psi([x_k,x_l], \alpha(x_1), \dots, \alpha(x_n)).
\end{align*}
The first term is zero, which can be seen using the module structure and the Hom-Jacobi identity.
It is also easy to see that $\delmm \delam \psi$ gives the same.
\end{proof}

So we can define: 
\begin{defn}
	The cohomology defined by the $\alpha$-type Chevalley-Eilenberg complex $(\aLC(\g,M),\del)$ is called $\alpha$-type  cohomology of $\g$ with values in $M$ and denoted by $\aLH(\g,M)$.
\end{defn}

One can of course specialize this to define the $\alpha$-type cohomology of a Hom-Lie algebra with values in itself, where the action is given by the adjoint action.
So let $(\g,\nu,\alpha)$ be a Hom-Lie algebra then we define 
\begin{equation}
	\aLC^n(\g) = \Hom(\Lambda^n \g,\g) \oplus \Hom(\Lambda^{n-1}\g,\g)
\end{equation}
and maps 
\begin{align}
	\begin{split}
	\delmm \phi(x_1,\dots,x_{n+1}) & = \sum_{i=1}^{n+1}  (-1)^{i+1}[\alpha^{n-1}(x_i), \phi(x_1,\dots,x_{n+1})]                                   \\
																	 & - \sum_{i<j}(-1)^{i+j-1} \phi([x_i,x_j],\alpha(x_1 ),\dots,\hat x_i,\hat x_j,\dots,\alpha(x_n)),                 
	\end{split}\\
	\begin{split}
	\delaa \psi(x_1,\dots,x_{n})   & = \sum_{i=1}^{n}  (-1)^{i+1}[\alpha^{n-1}(x_i), \psi(x_1,\dots,x_{n})]                                       \\
																	 & - \sum_{i<j}(-1)^{i+j-1} \psi([x_i,x_j],\alpha(x_1 ),\dots,\hat x_i,\hat x_j,\dots,\alpha(x_n)),                  
	\end{split}\\
	\delma \phi (x_1,\ldots, x_n)    & = \alpha( \phi(x_1,\ldots, x_n) ) - \phi(\alpha(x_1),\ldots,\alpha( x_n)),                                    \\
	\delam \psi(x_1,\ldots, x_{n+1}) & = \sum_{i\leq j} (-1)^{i+j-1} [[ \alpha^{n-2}(x_i),\alpha^{n-2}(x_j)],\psi(x_1,\dots,\hat x_i,\hat x_j,x_{n+1})] .
\end{align}
The differential $\del$ can now be defined as before by $\del(\phi,\psi) = (\delmm \phi -  \delam \psi , \delma \phi - \delaa \psi)$.

This complex is called $\alpha$-type Chevalley-Eilenberg complex of $\g$ with values in itself and its cohomology is denoted by $\aLH(\g,\g)$.

\begin{remark}\label{rm:symdiff}
	Note that since $\phi$ is completely skewsymmetric, we can write 
	\begin{align*}
		\delam \psi( x_1, \dots,x_{n+1}) &= \sum_{\sigma \in S_{n+1}} \sign(\sigma) \frac{1}{2 \cdot(n-1)!} \alpha^{n-2}([x_{\sigma(1)},x_{\sigma(2)}]) \psi(x_{\sigma(3)},\dots, x_{\sigma(n+1)}), \\
		\delmm \phi ( x_1, \dots,x_{n+1})&= \sum_{\sigma \in S_{n+1}} \sign(\sigma) \frac{1}{n!} \big(\alpha^{n-1}(x_{\sigma(1)}) \cdot \phi(x_{\sigma(2)}, \dots, x_{\sigma(n+1)}) \\
		& - \frac{1}{2} \phi([x_{\sigma(1)},x_{\sigma(2)}],x_{\sigma(3)},\dots, x_{\sigma(n+1)}) \big) ,\\
		\delaa \psi ( x_1, \dots,x_{n})&= \sum_{\sigma \in S_{n}} \sign(\sigma) \frac{1}{(n-1)!}  \big( \alpha^{n-1}(x_{\sigma(1)}) \cdot \psi(x_{\sigma(2)}, \dots, x_{\sigma(n)})\\
		& - \frac{1}{2} \psi([x_{\sigma(1)},x_{\sigma(2)}],x_{\sigma(3)},\dots, x_{\sigma(n)}) \big).
	\end{align*}	
\end{remark}

Let $A$ be a Hom-associative algebra and $M$ be an $A$-bimodule. We consider the $\alpha$-type Hochschild cohomology defined in \citep{homhcoho}. Then we can define a map $\Phi: \aHC^\bullet(A,M) \to \aLC^\bullet(A_L,M_L)$ by 
\begin{equation}
	\Phi(\phi)(x_1,\dots,x_n) = \sum_{\sigma \in S_n} \sign(\sigma) \phi(x_{\sigma(1)}, \dots,x_{\sigma(n)}).
\end{equation}
Here $S_n$ denotes the symmetric group and $\sign(\sigma)$ the signature of a permutation $\sigma$.

\begin{theorem}
	The map $\Phi$ is a surjective chain map. So it induces a map in cohomology. 
\end{theorem}
\begin{proof}
	It is clear that it is surjective.
	We have to check $\Phi \d = \del \Phi$, where $\d$ denotes the $\alpha$-type  Hochschild differential.
	Clearly $\Phi \d_{\mu\alpha} = \delma \Phi$. Further, we have
	\begin{align}
		\Phi \d_{\mu\mu} \phi(x_1, \dots, x_{n+1}) & =   \sum_{\sigma \in S_n} \sign(\sigma) \Big(\alpha^{n-1}( x_{\sigma(1)}) \phi(x_{\sigma(2)}, \dots,x_{\sigma(n+1)}) \\
		                                        & + (-1)^{n}\phi(x_{\sigma(1)}, \dots,x_{\sigma(n)}) \alpha^{n-1}( x_{\sigma(n+1)})                                        \\
		                                        & +\sum_i (-1)^i \phi( \alpha(x_{\sigma(1)}), \dots,x_{\sigma(i)}x_{\sigma(i+1)} ,\dots, \alpha(x_{\sigma(n+1)} ) )  \Big)         
		\\
		                                        & =  \sum_{i=1}^n \sum_{\sigma \in \tilde S_n} \sign(\sigma) \Big( \alpha^{n-1}(x_i) \cdot_L \phi(x_{\sigma(1)},\dots,\hat x_i \dots,x_{\sigma(n)})\\
																						& +\sum_i \frac{1}{2}\phi( \alpha(x_{\sigma(1)}), \dots,[x_{\sigma(i)},x_{\sigma(i+1)} ],\dots, \alpha(x_{\sigma(n+1)}) ) \Big),          
	\end{align}
where $\tilde S_n$ is the symmetric group on $\{1, \dots,\dots , n+1\} \setminus \{i\}$. 
	Using \Cref{rm:symdiff}  it is easy to see that this equals $\delmm \Phi$. The proof for the other parts of the differential are similar. 
\end{proof}

\begin{remark}
The complex of  Hom-Lie algebras cohomology given in \citep{homcoho} is a subcomplex  of the  $\alpha$-type Chevalley-Eilenberg complex  defined here. 
It is spanned by cocycles of the form $(\phi,0)$.  To get that $\del( \phi,0)$ is again of this form, we need $\delma \phi =0$. So we define the complex $\LC^n(A) = \{ \phi \in \Hom(\Lambda^n \g,g) |\alpha \phi = \phi \alpha^{\otimes n}\}$.  The map $\delmm$ defines a differential on this complex.  This is precisely the complex given in \citep{homcoho}.

\end{remark}

\subsection{Cohomology for Hom-Lie algebras of Lie-type}

In this section, we aim to study the cohomology of Hom-Lie algebras obtained by Yau twist of a Lie algebra together with an  endomorphism. We show a relationship between the cohomology of a Lie algebra morphism and the $\alpha$-type Chevalley-Eilenberg cohomology.

The cohomology for a homomorphism between two Lie algebras has been  studied in \citep{nijenhuismor, fregier1,fregier2}, see \citep{Arfa} for the Hom-case. Here we want to modify this to the case of an endomorphism.
So let $\g$ be a Lie algebra and $\gamma:\g \to \g$  be an endomorphism. Then the usual complex restricts to  
\begin{equation}
	\eC^n(\gamma):= \eC^n(\g,\gamma) :=  \eC^n_\nu(\g,\gamma) \oplus  \eC^n_\gamma(\g,\gamma) = \Hom(\Lambda^n \g, \g) \oplus \Hom(\Lambda^{n-1}\g,\g)
\end{equation}
for all $n  \in \Z$.
We denote an element in $\eC^n(\g,\gamma)$ by a pair $(\phi,\psi)$. The differential is defined by $\delta(\phi , \psi) = ( \del_{CE} \phi, -\del_{CE} \psi + \del_\gamma \phi)$, where $\del_{CE}$ is the usual Chevalley-Eilenberg differential and $\del_\gamma \phi = \gamma \phi - \phi \gamma^{\otimes k}$ for $\phi \in \Hom(\Lambda^k \g,\g)$.

We set $\Hom(\Lambda^n \g,\g) = 0$ for $n \leq 0$. So we also set $\Hom(\Lambda^0 \g,\g)$ to zero, since this way it agrees with the complex for the $\alpha$-type cohomology.

Now,  we regard the Hom-Lie algebra $\g_\gamma$ obtained by Yau twist of $\g$ by $\gamma$. We define a linear map $\Phi:C^\bullet(\g,\gamma) \to \aLC^\bullet(\g_\gamma)$ by 
\begin{equation}
	(\phi,\psi) \mapsto ( \gamma^{n-1} \phi +  \gamma^{n-2} \psi \circ \nu, \gamma^{n-2} \psi), 
\end{equation}
where 
$$(\psi \circ \nu)(x_1,\dots,x_n) = \psi(\nu \wedge \id^{\wedge n-2})(x_1,\dots,x_n)= \sum_{i < j}(-1)^{i+j-1} \psi([x_i,x_j], x_1, \dots, \hat x_i ,\hat x_j, \dots x_n)$$	
for $\phi \in C^n_\nu(\gamma)$ and $\psi \in C^n_\gamma(\gamma)$. Then we get the following 
\begin{theorem}\label{th:morph}
	Let $\g$ be a Lie algebra and $\gamma: \g \to \g$ be an endomorphism. Then $\Phi$ defined as above is chain map, which is an isomorphism if $\gamma$ is invertible. So in particular it induces a homomorphism in cohomology.
\end{theorem}
\begin{proof}
	One has to show that $\del \Phi (\phi ,\psi) =  \Phi \delta (\phi,\psi)$ for $(\phi, \psi) \in C(\g,\gamma)$, where $\delta$ denotes the differential in $C(\g,\gamma)$ and $\del$ the one in $\aLC(\g_\gamma)$.
	We have
	\begin{align*}
		\Phi(\delta \phi) &(x_1, \dots ,x_{n+1})= \Phi\Big( \sum_i \pm x_i \cdot \phi(x_1,\dots ,x_{n+1}) - \sum_{i<j}\pm \phi([x_i,x_j],x_1,\dots ,x_{n+1}), \\
		& \gamma \phi(x_1,\dots  ,x_{n+1}) - \phi(\gamma(x_1),\dots ,\gamma(x_{n+1}))\Big) \\
		=& \Big(\sum_i\gamma^n(x_i \cdot \phi(x_1,\dots, x_{n+1}) )- \sum_{i<j} \pm \gamma^n(\phi([x_i,x_j],x_1,\dots ,x_{n+1})) \\
		 & + \sum_{i<j} \gamma^n\phi([x_i,x_j],x_1,\dots, x_{n+1}) - \gamma^{n-1} \phi (\gamma([x_i,x_j]),\gamma(x_1),\dots , \gamma(x_{n+1})) , \\
		& \gamma^n \phi(x_1,\dots, x_{n+1}) - \gamma^{n-1} \phi (\gamma(x_1),\dots  , \gamma(x_{n+1})) \Big) \\
		=& \Big(\sum_i\gamma^n(x_i \cdot \phi(x_1,\dots ,x_{n+1}) )- \sum_{i<j}  \gamma^{n-1} \phi (\gamma([x_i,x_j]),\gamma(x_1),\dots ,\gamma(x_{n+1})) , \\ 
		& \gamma^n \phi(x_1,\dots ,x_{n+1}) - \gamma^{n-1} \phi (\gamma(x_1),\dots ,\gamma(x_{n+1})) \Big) \\
		=& \del( \gamma^{n-1} \phi(x_1, \dots ,x_{n+1}),0) = \del \Phi(\phi)(x_1,\dots ,x_{n+1}).
	\end{align*}
	Similarly one can check that $\del \Phi (0 ,\psi) =  \Phi \delta (0,\psi)$. If $\gamma$ is invertible the inverse of $\Phi$ is given by $\Phi^{-1}(\phi,\psi) = (\gamma^{-n+1} \phi - \gamma^{-n+1}\psi ,\gamma^{-n+2}\psi)$.
\end{proof}

We compute the $\alpha$-type cohomology of Lie algebras, viewed as Hom-Lie algebras. 
\begin{theorem}
	Let $(\g,\nu,\id)$ be a Lie algebra considered as a Hom-Lie algebra. Then 
	\begin{equation}
		\aLH^n(\g,\g) =  \LH^n(\g,\g) \oplus \LH^{n-1}(\g,\g).
	\end{equation}
	Here $\LH(\g,\g)$ denotes the ordinary Chevalley-Eilenberg cohomology of $\g$ but with  $\LH^1(\g,\g)$ replaced  by $\Der(\g)$ and $\LH^0(\g,\g)$ by $\{0\}$.
\end{theorem}
\begin{proof}
	This follows directly from \Cref{th:morph}, since $\del_\gamma =0$ and the other two parts of the differential are precisely the ordinary Chevalley-Eilenberg differential on $\LC^n(\g,\g)$ and $\LC^{n-1}(\g,\g)$.  One has to replace $\LH^1(\g,\g)$ with $\Der(\g)$ and $\LH^0(\g,\g)$ with $\{0\}$ since our complex starts with $\aLC^1(\g,\g)$. 
\end{proof}

\subsection{Whitehead  Theorem}

In this section we need  the field $\K$ to be of characteristic $0$.
If $\g$ is a simple Lie algebra, the well known Whitehead Lemma states that $\LH(\g,\g)$ is trivial. 
 So the $\alpha$-type Chevalley-Eilenberg cohomology consists of the derivations. 
 Simple Hom-Lie algebras are Yau twists of semi-simple Lie algebras by automorphisms \citep{chen}. In this section we compute the $\alpha$-type Chevalley-Eilenberg cohomology for (finite-dimensional) simple Hom-Lie algebras.

\begin{prop}\label{th:simplecoho}
	Let $\g$ be a (finite dimensional) semi-simple Lie algebra and $\gamma:\g \to \g$ an automorphism. Then $H^1(\gamma) = \Der_\gamma(\g)$, $H^2(\gamma)=  \factor{\gamma (\Der(\g))}{\del_\gamma (\Der(\g))}$ and $H^k(\gamma)= 0$ for $k \geq 2$.  Here $\Der_\gamma (\g)=  \{\phi \in \Der(\g) | \gamma \phi =\phi \gamma \}$ and $\alpha(\Der(\g)) = \{ \alpha \circ \phi| \phi \in \Der(\g)\}$.
\end{prop}
\begin{proof}
	By the usual Whitehead Lemma we get that $\LH(\g,M) = \{0\}$ for a simple $\g$-module $M$. So in particular $\LH(\g,\g) =0$ and $\LH(\g,\tilde \g)=0$, where $\tilde \g$ denotes $\g$ with  the action given by $x \cdot y = [\gamma(x) ,y] $ for $x \in \g$ and 
	$y \in \tilde\g$.  So if we consider the spectral sequence associated to the vertical filtration of the bicomplex $C(\gamma)$, for the first page of it, which is the cohomology with respect to $\del_{CE}$,  we get 
	\begin{equation}
	 H^1(C^\bullet_\nu(\g,\g)) = \Der(\g) = \operatorname{InnDer}(\g),   H^1(C^\bullet_\gamma(\g)) = \Der(\tilde \g) = \operatorname{InnDer}(\tilde\g).
	\end{equation}
	We do not get zero here, since we started the complex for $H(\gamma)$ by $\Hom(\g,\g)$ and dropped $\Hom(\K,\g)$. So the second page of the spectral sequence gives the claimed result. We have $\operatorname{\alpha-Der}(\g) = \alpha (\Der(\g))$, so every $\gamma$-derivation is of the form $\gamma \phi$ for a derivation $\phi$. 
\end{proof}
 
If $\alpha$ is diagonalizable, we further have $H^2(\gamma)=  \factor{\gamma \Der(\g)}{\del_\gamma \Der(\g)} \cong \gamma \Der_\gamma(\g)$. 

\begin{theorem}[Whitehead Theorem for Hom-Lie algebras]
	Let $(\g,\nu,\alpha)$ be a finite dimensional simple Hom-Lie algebra, then $\aLH^1(\g,\g) = \Der_\alpha(\g)$, $\aLH^2(\g,\g)=  \factor{\alpha(\Der(\g))}{\del_\gamma(\Der(\g))}$ and $\aLH^k(\g,\g)= 0$ for $k \geq 2$. 
\end{theorem}
\begin{proof}
	For a simple Hom-Lie algebra we have that $\alpha$ is invertible and  in \citep{chen} it is proven that $\g_{\alpha^{-1}}$ is a semi-simple Lie algebra. So it is enough to compute $H(\g_{\alpha^{-1}},\alpha)$.  This is given by the previous proposition.
\end{proof}

\section{\texorpdfstring{$L_\infty$}{L infinity}-structure}\label{sc:linfty}

As in the Hom-associative case, we conjecture that there is an $L_\infty$-structure on the complex $\aLC(\g,\g)$, where $\g$ is a vector space considered as Hom-Lie algebra with the zero Hom-Lie  structure, such that Maurer-Cartan elements are Hom-Lie algebra structures on $\g$ and the cohomology one can construct from this is the one we defined in the previous section. In fact we suspect that the $L_\infty$-structure describing Hom-Lie algebras can be derived from the one describing Hom-associative algebras by total antisymmetrization. 

However using a graph complex, which corresponds to a free operad, we are able to calculate the low degrees.

We consider the free symmetric operad,  spanned by  totally skewsymmetric   operations $\verm2, \verm3, \verm4, \vera1, \vera2, \vera3$.  We consider them to be graded with $\deg(\verak) = k-1$ and $\deg(\vermk) = k-2$. 

On this free operad we define a  differential on generators by $\del \verm2 = \del \vera1 =0$, 
\begin{align*}
	\del \verm3 & = 
	\begin{tikzpicture}[scale=0.3,point/.style={draw,shape=circle,fill=blue,minimum size=2,inner sep=0}, left/.style={draw,regular polygon, regular polygon sides=3, rotate=90,minimum size=5,inner sep=0},
	right/.style={draw,regular polygon, regular polygon sides=3, rotate=-90,minimum size=5,inner sep=0},
	circ/.style={draw,shape=circle,minimum size=5,inner sep=0}]
	\node [point] (a0) at (2.25,2) {};
	\node [point] (a1) at (1.5,1) {};
	\draw (a0) -- +(0,0.7) ;
	\draw (1,0) -- (a1);
	\draw (2,0) -- (a1);
	\draw  (a1) -- (a0); 
	\node [circ] (a2) at (3,1) {};
	\draw (3,0) -- (a2);
	\draw  (a2) -- (a0); 
	\end{tikzpicture}
	,\; \del \vera2 =
	\begin{tikzpicture}[scale=0.3,point/.style={draw,shape=circle,fill=blue,minimum size=2,inner sep=0}, left/.style={draw,regular polygon, regular polygon sides=3, rotate=90,minimum size=5,inner sep=0},
	right/.style={draw,regular polygon, regular polygon sides=3, rotate=-90,minimum size=5,inner sep=0},
	circ/.style={draw,shape=circle,minimum size=5,inner sep=0}]
	\node [circ] (a0) at (1.5,2) {};
	\node [point] (a1) at (1.5,1) {};
	\draw (a0) -- +(0,0.7) ; 	\draw (1,0) -- (a1);
	\draw (2,0) -- (a1);
	\draw  (a1) -- (a0); 
	\end{tikzpicture}
	-\begin{tikzpicture}[scale=0.3,point/.style={draw,shape=circle,fill=blue,minimum size=2,inner sep=0}, left/.style={draw,regular polygon, regular polygon sides=3, rotate=90,minimum size=5,inner sep=0},
	right/.style={draw,regular polygon, regular polygon sides=3, rotate=-90,minimum size=5,inner sep=0},
	circ/.style={draw,shape=circle,minimum size=5,inner sep=0}]
	\node [point] (a0) at (1.5,2) {};
	\node [circ] (a1) at (1,1) {};
	\draw (a0) -- +(0,0.7) ; 	\draw (1,0) -- (a1);
	\draw  (a1) -- (a0); 
	\node [circ] (a2) at (2,1) {};
	\draw (2,0) -- (a2);
	\draw  (a2) -- (a0); 
	\end{tikzpicture}, \\
	\del \verm4 & = 
	\begin{tikzpicture}[scale=0.3,point/.style={draw,shape=circle,fill=blue,minimum size=2,inner sep=0}, left/.style={draw,regular polygon, regular polygon sides=3, rotate=90,minimum size=5,inner sep=0},
	right/.style={draw,regular polygon, regular polygon sides=3, rotate=-90,minimum size=5,inner sep=0},
	circ/.style={draw,shape=circle,minimum size=5,inner sep=0}]
	\node [point] (a0) at (2.83333,2) {};
	\node [point] (a1) at (1.5,1) {};
	\draw (a0) -- +(0,0.7) ; 	\draw (1,0) -- (a1);
	\draw (2,0) -- (a1);
	\draw  (a1) -- (a0); 
	\node [circ] (a2) at (3,1) {};
	\draw (3,0) -- (a2);
	\draw  (a2) -- (a0); 
	\node [circ] (a3) at (4,1) {};
	\draw (4,0) -- (a3);
	\draw  (a3) -- (a0); 
	\end{tikzpicture}
	- \begin{tikzpicture}[scale=0.3,point/.style={draw,shape=circle,fill=blue,minimum size=2,inner sep=0}, left/.style={draw,regular polygon, regular polygon sides=3, rotate=90,minimum size=5,inner sep=0},
	right/.style={draw,regular polygon, regular polygon sides=3, rotate=-90,minimum size=5,inner sep=0},
	circ/.style={draw,shape=circle,minimum size=5,inner sep=0}]
	\node [point] (a0) at (3,3) {};
	\node [point] (a1) at (2,1) {};
	\draw (a0) -- +(0,0.7) ; 	\draw (1,0) -- (a1);
	\draw (2,0) -- (a1);
	\draw (3,0) -- (a1);
	\draw  (a1) -- (a0); 
	\node [circ] (a2) at (4,2) {};
	\node [circ] (a3) at (4,1) {};
	\draw (4,0) -- (a3);
	\draw  (a3) -- (a2); 
	\draw  (a2) -- (a0); 
	\end{tikzpicture}
	+ \begin{tikzpicture}[scale=0.3,point/.style={draw,shape=circle,fill=blue,minimum size=2,inner sep=0}, left/.style={draw,regular polygon, regular polygon sides=3, rotate=90,minimum size=5,inner sep=0},
	right/.style={draw,regular polygon, regular polygon sides=3, rotate=-90,minimum size=5,inner sep=0},
	circ/.style={draw,shape=circle,minimum size=5,inner sep=0}]
	\node [point] (a0) at (2.5,3) {};
	\node [circ] (a1) at (1.5,1) {};
	\draw (a0) -- +(0,0.7) ; 	\draw (1,0) -- (a1);
	\draw (2,0) -- (a1);
	\draw  (a1) -- (a0); 
	\node [point] (a2) at (3.5,2) {};
	\node [circ] (a3) at (3,1) {};
	\draw (3,0) -- (a3);
	\draw  (a3) -- (a2); 
	\node [circ] (a4) at (4,1) {};
	\draw (4,0) -- (a4);
	\draw  (a4) -- (a2); 
	\draw  (a2) -- (a0); 
	\end{tikzpicture}, \\
	\del \vera3 & = 
	\begin{tikzpicture}[scale=0.3,point/.style={draw,shape=circle,fill=blue,minimum size=2,inner sep=0}, left/.style={draw,regular polygon, regular polygon sides=3, rotate=90,minimum size=5,inner sep=0},
	right/.style={draw,regular polygon, regular polygon sides=3, rotate=-90,minimum size=5,inner sep=0},
	circ/.style={draw,shape=circle,minimum size=5,inner sep=0}]
	\node [circ] (a0) at (2,2) {};
	\node [point] (a1) at (2,1) {};
	\draw (a0) -- +(0,0.7) ; 	\draw (1,0) -- (a1);
	\draw (2,0) -- (a1);
	\draw (3,0) -- (a1);
	\draw  (a1) -- (a0); 
	\end{tikzpicture}
	- \begin{tikzpicture}[scale=0.3,point/.style={draw,shape=circle,fill=blue,minimum size=2,inner sep=0}, left/.style={draw,regular polygon, regular polygon sides=3, rotate=90,minimum size=5,inner sep=0},
	right/.style={draw,regular polygon, regular polygon sides=3, rotate=-90,minimum size=5,inner sep=0},
	circ/.style={draw,shape=circle,minimum size=5,inner sep=0}]
	\node [point] (a0) at (2,2) {};
	\node [circ] (a1) at (1,1) {};
	\draw (a0) -- +(0,0.7) ; 	\draw (1,0) -- (a1);
	\draw  (a1) -- (a0); 
	\node [circ] (a2) at (2,1) {};
	\draw (2,0) -- (a2);
	\draw  (a2) -- (a0); 
	\node [circ] (a3) at (3,1) {};
	\draw (3,0) -- (a3);
	\draw  (a3) -- (a0); 
	\end{tikzpicture}
	- \begin{tikzpicture}[scale=0.3,point/.style={draw,shape=circle,fill=blue,minimum size=2,inner sep=0}, left/.style={draw,regular polygon, regular polygon sides=3, rotate=90,minimum size=5,inner sep=0},
	right/.style={draw,regular polygon, regular polygon sides=3, rotate=-90,minimum size=5,inner sep=0},
	circ/.style={draw,shape=circle,minimum size=5,inner sep=0}]
	\node [circ] (a0) at (2.25,2) {};
	\node [point] (a1) at (1.5,1) {};
	\draw (a0) -- +(0,0.7) ; 	\draw (1,0) -- (a1);
	\draw (2,0) -- (a1);
	\draw  (a1) -- (a0); 
	\node [circ] (a2) at (3,1) {};
	\draw (3,0) -- (a2);
	\draw  (a2) -- (a0); 
	\end{tikzpicture}
	- \begin{tikzpicture}[scale=0.3,point/.style={draw,shape=circle,fill=blue,minimum size=2,inner sep=0}, left/.style={draw,regular polygon, regular polygon sides=3, rotate=90,minimum size=5,inner sep=0},
	right/.style={draw,regular polygon, regular polygon sides=3, rotate=-90,minimum size=5,inner sep=0},
	circ/.style={draw,shape=circle,minimum size=5,inner sep=0}]
	\node [point] (a0) at (2.25,3) {};
	\node [circ] (a1) at (1.5,1) {};
	\draw (a0) -- +(0,0.7) ; 	\draw (1,0) -- (a1);
	\draw (2,0) -- (a1);
	\draw  (a1) -- (a0); 
	\node [circ] (a2) at (3,2) {};
	\node [circ] (a3) at (3,1) {};
	\draw (3,0) -- (a3);
	\draw  (a3) -- (a2); 
	\draw  (a2) -- (a0); 
	\end{tikzpicture}.
\end{align*}
Here one has to take the total antisymmetrization of each term, but only take terms, which are not equal using the antisymmetry of the operations.

Using the same technique as in \citep{homhcoho}, based on \citep{MR2203621,fregier1}, we obtain an $L_\infty$-structure on $\aLC(\g,\g)$ by replacing the nodes with maps  in the complex.

We give the corresponding brackets for $\phi_i \in \aC_\nu^i(\g,\g), \psi_i \in \aC_\alpha^i(\g,\g), \alpha_i \in \aC_\alpha^2(\g,\g)$ and $\nu_i \in \aC_\nu^2(\g,\g)$:
\begin{align*}
	\deg 1: \\
	[\nu_1 ,\nu_2, \alpha]_\nu                 & = \nu_1 (\nu_2 \wedge \alpha)                                                                                
	+ \nu_2 (\nu_1 \wedge \alpha), \\
	[\nu,\alpha]_\alpha                        & = \alpha \nu  ,                                                                                               \\
	[\nu, \alpha_1,\alpha_2]_\alpha            & =  - \nu( \alpha_1 \wedge \alpha_2 )    ,                                                                     \\
	\deg 2: \\
	[\phi_3, \nu,\alpha_1,\alpha_2]_\nu        & =  \phi_3 (\nu \wedge \alpha_1 \wedge \alpha_2)                                                              
	- \nu(\phi_3 \wedge \alpha_1 \alpha_2 )   - \nu(\phi_3 \wedge \alpha_2 \alpha_1 ) ,\\
	[\psi_3,\nu_1,\nu_2,\alpha_1,\alpha_2]_\nu     & =  \sum_{\sigma \in S_2} \nu_{\sigma(1)}(\psi_3 \wedge \nu_{\sigma(2)}(\alpha_{1} \wedge \alpha_{2})) ,\\
	[\phi_3,\alpha]_\alpha                     & = \alpha \phi_3      ,                                                                                        \\
	[\phi_3,\alpha_1,\alpha_2,\alpha_3]_\alpha & =  - \phi_3 (\alpha_{1} \wedge \alpha_{2} \wedge \alpha_{3} )                                       ,         \\
	[\psi_3,\nu,\alpha]_\alpha                 & = \psi_3 ( \alpha \wedge \nu) ,                                                                               \\
	[\psi_3,\nu,\alpha_1,\alpha_2]_\alpha      & =  - \nu(\psi_3 \wedge \alpha_1 \alpha_2) - \nu(\psi_3 \wedge \alpha_2 \alpha_1)  .                           
\end{align*}

\begin{theorem}
	The Maurer-Cartan elements of this $L_\infty$-structure are Hom-Lie algebras, and the differential on $\aLC^2$ and $\aLC^3$ comes from it.
\end{theorem}
\begin{proof}
	It is clear, by regarding the  brackets in degree 0, that a Maurer-Cartan element is a Hom-Lie algebra.  The differentials of $(\phi,\psi) \in \aLC(\g,\g)$ can be computed by $\del\phi = [\alpha,\dots,\nu,\dots,\phi]$ and similar for $\psi$. Here $\alpha,\ldots$ stands for the insertion of zero or more $\alpha$ and the same for $\nu$.  Inspection of the brackets shows that this agrees with the definition given before.
\end{proof}

\section{Deformation theory including the structure map}\label{sc:deformation}

In this section we briefly discuss how the $\alpha$-type Chevalley-Eilenberg cohomology can be used to study one-parameter formal deformations of a	 Hom-Lie algebra $(\g,\nu,\alpha)$, where the bracket $\nu$ and the structure map $\alpha$ are deformed. 
 For this we denote by $\K\ph$ the ring of formal power series over $\K$ in a formal parameter $t$. For a vector space $\g$ we denote by $\g\ph$ the space of formal power series of the form $\sum_{i=0}^\infty g_i t^i$ for $g_i \in \g$. 
This is  obviously a $\K\ph$-module. We recall that a $\K\ph$-linear map $\phi$ between $V\ph$ and $W\ph$ for two vector spaces $V,W$ can be given by  $\sum_{i=0}^\infty \phi_i t^i$, where $\phi_i: V \to W$ are $\K$-linear maps extended to $\K\ph$-linear  maps $V\ph \to W\ph$ in the obvious way.    

\begin{defn}
	 A (one-parameter formal) deformation of a Hom-Lie algebra $(\g,\nu,\alpha)$ over $\K$ is a Hom-Lie algebra $(\g\ph, \nu_t, \alpha_t)$ over $\K\ph$, such that  $\alpha_t = \alpha + \sum_{i=1}^\infty \alpha_i t^i$ and 
	$[ x,y]_t = [x,y] +  \sum_{i=1}^\infty t^i [x,y]_i $. Here we denote $\nu_t( x \otimes y) = [x,y]_t$. 
\end{defn}

\begin{defn}
	Two deformations $(\nu_t,\alpha_t)$ and $(\nu'_t,  \alpha'_t)$ are said equivalent if there exists a formal morphism $S= \id + \sum_{i=1}^\infty S_i t^i$, where $S_i: \g \to \g$ are linear maps, such that $S( [x,y]'_t) = [S(x),S(y)]_t$ and $S  \alpha'_t = \alpha_t S$. 
\end{defn}

It is clear that this  equivalence is an equivalence relation on the set of deformations of a Hom-Lie algebra.

Let  $\g$ be a Lie algebra viewed as Hom-Lie algebra and $(\g\ph, \nu_t,\alpha_t)$ be a deformation. Then $\alpha_1$ is a derivation. 
This follows from the multiplicativity, using the expansions $\alpha_t([x,y]_t) = [x,y] + t( \alpha_1 [x,y] + [x,y]_1 ) +\O(t^2) $ and $[\alpha_t(x),\alpha_t(y)]_t= [x,y] + t ([\alpha_1 x, y] + [ x, \alpha_1  y] + [x,y]_1) +\O(t^2)$.

The condition that $(\g\ph,\nu_t,\alpha_t)$ is a deformation of $(\g,\nu,\alpha)$ is equivalent to the following two conditions:

First it has to satisfy the Hom-Jacobi identity, which gives 
\begin{equation}
	\sum_{i,j,k \geq 0} \cycl_{x,y,z} \nu_i(\alpha_j(x),\nu_k(y,z)) t^{i+j+k} =0.
\end{equation}
Here and in the following equations $ \cycl_{x,y,z}$ denotes the cyclic sum over the elements $x,y,z \in \g$.
At order $n$ in $t$ this gives 
\begin{equation}
\sum_{\substack{ i,j,k\geq 0 \\ i+j+k=n }} \cycl_{x,y,z} \nu_i(\alpha_j(x),\nu_k(y,z))  =0.
\end{equation}
This is called the $n$-th deformation equation with respect to the Hom-Jacobi identity.
It can be rearranged to 
\begin{equation} 
	(\delmm \nu_n + \delam \alpha_n)(x,y,z)= \sum_{\substack{ 0 \leq i,j,k \leq n-1 \\ i+j+k=n }} \cycl_{x,y,z} \nu_i(\alpha_j(x),\nu_k(y,z))  =0 ,
\end{equation}
where $\delmm \nu_n(x,y,z)= -  \cycl_{x,y,z} \nu_n([x,y],\alpha(z)) +  \cycl_{x,y,z} [\alpha(x),\nu_n(y,z)]$ and $\delam \alpha_n (x,y,z) =  \cycl_{x,y,z} [\alpha_n(x),[y,z]]$ are the parts of the differential defined in \cref{sc:coho}. We denote the right hand side by $R_n^1$.

Second,  it has to satisfy the multiplicativity, this is 
\begin{equation}
 \sum_{i,j\geq 0} \alpha_i(\nu_j(x,y)) t^{i+j}-  \sum_{i,j,k\geq 0} \nu_i(\alpha_j(x),\alpha_k(y)) t^{i+j+k} =0.
\end{equation}
At order $n$  this gives 
\begin{equation}
	\sum_{i=0}^{n} \alpha_i(\nu_{n-i}(x,y)) -  \sum_{\substack{0 \leq i,j,k \leq n \\ i+j+k =n } } \nu_i(\alpha_j(x),\alpha_k(y))  =0.
\end{equation}
We call this the $n$-th deformation equation with respect to the multiplicativity. Again this can be rewritten as 
\begin{equation}
	(\delaa \alpha_n + \delma \nu_n)(x,y) = -\sum_{i=0}^{n-1} \alpha_i(\nu_{n-i}(x,y)) +  \sum_{\substack{0 \leq i,j,k \leq n-1 \\ i+j+k =n } } \nu_i(\alpha_j(x),\alpha_k(y)), 
\end{equation}
where $\delma \nu_n = \alpha \nu_n - \nu_n (\alpha \otimes \alpha)$ and  $\delaa \alpha_n(x,y) = [\alpha(x), \alpha_n(y)] -[\alpha(y), \alpha_n(x)]    - \alpha_n([x,y])$.  We denote the right hand side by $R_n^2$. 

Since the deformation is governed by an $L_\infty$-algebra, we have the usual statement relating deformations and cohomology. We will omit the proof here since these follows from the deformation equations given above and  are almost identical to the Hom-associative case given in \citep{homhcoho}.

\begin{theorem}
	Let $(\g,\nu,\alpha)$ be a Hom-Lie algebra and $(\g\ph,\nu_t,\alpha_t)$ be a deformation of $\g$. Then we have 
	\begin{enumerate}
		\item The first order term is a 2-cocycle, i.e.\ we have $\del(\nu_1,\alpha_1) = 0$, whose cohomology class is invariant under equivalence.
		\item The n-th deformation equations with respect to the Hom-Jacobi identity and multiplicativity,  are equivalent to $\del(\nu_n,\alpha_n) = (R^1_n,R^2_n)$.
		Furthermore,  $(R^1_n,R^2_n)$ is a 3-cocycle, i.e. $\del (R^1_n,R^2_n) = 0$. 
	\end{enumerate}
\end{theorem}

This means that the construction of a deformation order by order gives equations in $\aLH^3(\g,\g)$, we also get: 
\begin{corollary}
	If $\aLH^3(\g,\g)=0$ any  finite deformation up to order $n$ can be extended to a full deformation. So especially every 2-cocycle can be extended to a deformation.
\end{corollary}

\begin{prop}
	If two deformations $(\nu_t, \alpha_t)$ and $({\nu'}_t,{\alpha'}_t)$ agree up to order $n-1$, i.e. $\alpha_i = \alpha'_i, \nu_i= \nu'_i$ for $i=1,\dots,n-1$, then we have $\del(\nu_n - {\nu'}_n , \alpha_n - {\alpha'}_n) =0$ and there exists an equivalence up to order $n$ if there exists  linear maps $S_n: \g \to \g$ such that $\del(S_n,0)= (\nu - {\nu'}_n , \alpha_n - {\alpha'}_n)$.
\end{prop}

So the construction of an equivalence is a problem in $\aLH^2(\g,\g)$, and we get 
\begin{corollary}
	If $\aLH^2(\g,\g) = 0$, then all deformations are trivial, i.e.\ are equivalent to the undeformed algebra. 
\end{corollary}

\section{Examples}

Next, we  compute the $\alpha$-type Chevalley-Eilenberg cohomology explicitly for some low dimensional Hom-Lie algebras by using a computer software. 

First we consider a Hom-Lie algebra, which is not of Lie-type.  Let $\g$ be a vector space with   basis  given by $(x,y,z)$. Then we define $\alpha$ by $\alpha(x)=x, \alpha(y) =y, \alpha(z)=0$ and a bracket by 
\begin{equation}
	[x,y]  = x, \; [y,z]  = z \text{ and } [x,z] = z. 
\end{equation}
It is easy to verify that this is in fact a Hom-Lie algebra and it is not of Lie-type since $z \in \img \nu$ but $z \notin \img(\alpha)$.
The following table gives the dimensions of the cohomology and of the coboundaries and cocycles:
\begin{center}
	\begin{tabular}{c|c|c|c|c}
		$i$ & $\dim \aLC^i$ & $\dim \img \del^i$ & $\dim \ker \del^i$ & $\dim \aLH^i$ \\
		\hline
		1   & 9             & 8                 & 1                  & 1             \\
		2   & 18            & 8                 & 10                 & 2             \\
		3   & 12            & 2                 & 10                 & 2             \\
		4   & 3             & 0                 & 3                  & 1             
	\end{tabular}
\end{center}

The cohomology space can be spanned by the following maps, where $\phi \in \aC^n_\nu(\g,\g) = \Hom(\Lambda^{n}\g,\g)$ and $\psi \in \aC^n_\alpha(\g,\g) = \Hom(\Lambda^{n-1}\g,\g)$:
\begin{align*}
	\aLH^1(\g,\g) : & \phi(z) = \lambda z     ,                                                       \\
	\aLH^2(\g,\g) : & \phi \equiv 0, \psi(z)= \lambda_1 z , g(x) = (\lambda_1+ \lambda_2) x , \psi(y) = \lambda_2 y, \\
	\aLH^3(\g,\g) : &  \phi \equiv 0,\psi(x,z) = \lambda_1(z) x  , \psi( y,z) = \lambda_2 z ,                       \\
	\aLH^4(\g,\g) : &  \phi \equiv 0,\psi(x,y,z) =  \lambda z        .                                               
\end{align*}

In the following  we consider a Hom-Lie algebra, which is of Lie type. For this we consider a vector space $\g$ spanned by $x,y,z$ and a Lie bracket defined by 
\begin{equation}
	[x,y] = x , [x,z] =x, [y,z] = y - z.
\end{equation}
The dimensions of the cohomology spaces are as  follows:
\begin{center}
	\begin{tabular}{c|c|c|c|c}
		$i$ & $\dim \aLC^i(\g,\g)$ & $\dim \img \del^i$ & $\dim \ker \del^i$ & $\dim \aLH^i(\g,\g)$ \\
		\hline
		1   & 9             & 3                 & 6                  & 6             \\
		2   & 18            & 6                 & 12                 & 9             \\
		3   & 12            & 3                 & 9                 & 3            \\
		4   & 3             & 0                 & 3                  & 0             
	\end{tabular}
\end{center}
It is enough to give its ordinary Chevalley-Eilenberg cohomology to know the $\alpha$-type Chevalley-Eilenberg cohomology. It is given by
\begin{align*}
	\LH^1(\g,\g) :& \phi(x) = \lambda_1x - \lambda_2 y + \lambda_2 z \\
	 & \phi(y) = \lambda_3 x - \lambda_4 y + \lambda_4 z \\
	 & \phi(z) = \lambda_5 x - \lambda_6 y + \lambda_6 z \\
	 \LH^2(\g,\g) :& \phi(x,y) = \lambda_1 z, \phi(x,z) = - \lambda_1 y  \\
	 	&\phi(y,z) = \lambda_2 x + \lambda_3z - \lambda_3 y,	 
\end{align*}
where $\lambda_i$ are parameters.

It is easy to see that $\gamma(x) =0, \gamma(y) =y , \gamma(z)=z$ defines a morphism of $\g$.
So the Yau twist $\g_\gamma$ is  a Hom-Lie algebra, with structure map $\gamma$ and the only non vanishing bracket is $[y,z]_\gamma= y-z$.
The dimensions  of the cohomology are given in the following table:
\begin{center}
	\begin{tabular}{c|c|c|c|c}
		$i$ & $\dim \aLC^i(\g_\gamma,\g_\gamma)$ & $\dim \img \del^i$ & $\dim \ker \del^i$ & $\dim \aLH^i(\g_\gamma,\g_\gamma)$ \\
		\hline
		1   & 9             & 6                 & 3                  & 3             \\
		2   & 18            & 7                 & 11                 & 5             \\
		3   & 12            & 2                 & 10                 & 3            \\
		4   & 3             & 0                 & 3                  & 1             
	\end{tabular}
\end{center}
The derivations and $\alpha$-derivations of $\g_\gamma$ can be obtained from the derivations of $\g$, which are compatible with $\gamma$. They are both given by:
\begin{equation}
	f(x) = \lambda_1 x,f(y) = \lambda_2 y -\lambda_2 z,f(z) = \lambda_3 y -\lambda_3 z.
\end{equation}
So they can be seen as part of the first and second cohomology.
The remaining cohomology is spanned by 
\begin{align}
	\aLH^2(\g_\gamma,\g_\gamma) &: \phi(x,y) =\lambda_1 x , \phi( x,z) =\lambda_2 x, \\
	\aLH^3(\g_\gamma,\g_\gamma) &: \phi(x,y,z) =\lambda_1 x , \psi( x,y) =\lambda_2 x, \psi( x,z) =\lambda_3 x, \\ 
	\aLH^4(\g_\gamma,\g_\gamma) &: \psi(x,y,z) =\lambda_1 x .
\end{align}
This cannot be obtained using \cref{th:morph}, since $x \notin \img \gamma$. 
So the cohomology of $\g$ and $\g_\gamma$ differ and most of the cohomology from $\g_\gamma$ cannot be obtained from the one of $\g$. Also all cocycles of $\g$ of degree $2$ or higher are either not compatible with $\gamma$ or map to zero under $\Phi$ and so do not contribute the the cohomology of $\g_\gamma$.

\section{Grand Crochet}\label{sc:grandcrochet}

It is well known that the Grassmann algebra $\Lambda^\bullet \g = S(\g[1])$  of a vector space $\g$ is a Hopf algebra, with coproduct given by $\Delta(x) = 1\otimes x + x \otimes 1$ for $x \in \g$, and extended to the rest of $\Lambda^\bullet \g$ as algebra morphism. By $\g[1]$ we mean the shifted space. This means an element $x\in \g$ has degree 1, since we assume $\g$ to be concentrated in degree 0.  
So maps on $\Lambda^\bullet \g$ are graded and we use the Koszul sign rule, i.e. $(\phi \otimes \psi)(x\otimes y) = (-1)^{\deg (\psi) \deg(x)} \phi(x) \otimes \psi(y). $

The structure maps $\alpha$ and $\beta$ can be extended to Hopf algebra morphism on $\Lambda^\bullet \g$, which we also denote by $\alpha$ and $\beta$ respectively.
 
We give a generalization to Hom-Lie algebras of the grand crochet defined by Lecomte and Roger \citep{lecomte}. For the case $\beta=\alpha^{-1}$ this was already done in \citep{shengbb}.

Let $\phi \in \Hom(\Lambda^\bullet\g,\Lambda^\bullet\g)$ for a vector space $\g$ with two structure maps $\alpha,\beta: \g \to \g$ which commute. Normally $\g$ will be a Lie bialgebra and one can regard  $\g$ here as a Hom-Lie bialgebra with trivial bracket and cobracket.

We define the  $\alpha$- and $\beta$-height of $\phi$, as integers $\abs{\phi}_\alpha$ and $\abs{\phi}_\beta$ respectively. We also write $\abs{\phi} = (\abs{\phi}_\alpha,\abs{\phi}_\beta)$. 
The height of $\alpha$ is $(1,0)$ and the height of $\beta$ is $(0,1)$. This explains the name. Further if $\g$ is a Lie bialgebra we set $\abs{\nu} = (1,0)$ and $\abs{\delta}=(0,1)$. 
More generally $\abs{\phi} = (i-1,j-1)$ for $\phi \in  \Hom(\Lambda^\bullet\g,\Lambda^\bullet\g)$ with the usual action and coaction, so for example in $\LC^{i,j}(\g)$. 
If the action is twisted by $\alpha^k$ this is $a \cdot x = [\alpha^k(a),x]$ and the coaction by $\beta^l$ then $\abs{\phi} = ( i+k-1,j+l-1)$. So e.g.\ $\abs{\phi}=(i,0)$ for $\phi \in \aC^i_\alpha(\g)$.   
On the other hand if we say $\phi$ has a certain height, we also assume the actions to be twisted accordingly. 

Given  maps $\phi_i$ and $\psi_i$  with height such that    $\abs{\phi_i} = \abs{\phi_j}$ and   $\abs{\psi_i} = \abs{\psi_j}$ for all $i,j$.  We set $\abs{(\phi_1 \otimes \dots \otimes \phi_k) (\psi_1 \otimes \dots \otimes \psi_l)}= \abs{\phi_i} + \abs{\psi_i}$.

More generally given an arbitrary composition of maps with height one can follow the paths from the inputs to the outputs and associate a height to them by adding the heights of the maps one passes. If all paths have the same height we call it homogeneous.

We set the height of the  product and coproduct on $\Lambda^\bullet \g$ to zero. By $\pr_\g$ or simply $\pr$ we denote the projection from $\Lambda^\bullet \g \to \g$. Note that $ \pr \mu = \mu(\pr \otimes \id + \id \otimes \pr )$.

\begin{remark}
	Note that all maps in the differentials, the compatibility condition and so on for Hom-algebras of any type we considered  are homogeneous. 
	Also most compositions can be made homogeneous be inserting $\alpha$ and $\beta$ as needed. This is even how most definitions can be obtained. One considers  the non-Hom case and adds the structure maps as needed. However it can happen that $\alpha^{-1}$, which has $\abs{\alpha^{-1}}=(-1,0)$, would be needed.
\end{remark}

Now we can define the grand crochet.

Let $\phi, \psi \in \Hom(\Lambda^\bullet \g,\Lambda^\bullet \g)$ be maps with height, then we define a product by
\begin{equation}
	\psi \circ \phi = \mu (\psi \otimes \alpha^\psi) (\mu \otimes \id) (\id \otimes \pr \otimes \id) (\id \otimes \Delta) (\alpha^\phi \otimes \phi) \Delta
\end{equation}
and with that 
\begin{equation}\label{eq:bb}
	\bb{\phi ,\psi} = \phi \circ \psi - (-1)^{\deg(\phi)\deg(\psi)} \psi \circ \phi.
\end{equation}
Here we used the shorthand $\alpha^\phi = \alpha^{\abs{\phi}_\alpha}\beta^{\abs{\phi}_\beta}$, which has the same height as $\phi$.

For $\phi \in \Hom(\Lambda^k\g,\g)$ and $\psi \in \Hom(\Lambda^l \g,\g)$  the  product $\circ$ can be written as 
\begin{multline}
	(\phi \circ \psi)(x_1, \dots, x_{k+l-1}) =  \frac{1}{k!(l-1)!} \sum_{\sigma \in S_{n+k}}  \phi( \psi( x_{\sigma (1)} ,\dots, x_{\sigma(k)}), \\ \alpha^{\abs{\phi}_\alpha}(x_{\sigma(k+1)}) ,\dots,\alpha^{\abs{\phi}_\alpha}(x_{\sigma(k+l-1)})).
\end{multline}

\begin{prop}\label{th:bbdif}
	For  $(\phi,\psi) \in \aLC(\g,\g) = \Hom(\Lambda^\bullet \g, \g)$  and $\chi_1, \chi_2 \in \Hom(\Lambda^\bullet\g,\Lambda^\bullet \g)$ with arbitrary heights we have:
	\begin{itemize}
		\item $\delmm \phi =(-1)^{k-1} \bb{\nu,\phi}$. 
		\item $\delaa \psi =(-1)^{k-1} \bb{\nu,\psi}$. 
		\item $\delma (\chi_1 \circ \chi_2) =  (\delma \chi_1 \circ \chi_2)+ (\chi_1 \circ \delma \chi_2)$
	\end{itemize}
\end{prop}
\begin{proof}
	We have that 
	\begin{equation}
		\nu \circ \phi = (-1)^{k-1}[ \alpha^k( x^{(1)}), \phi(x^{(2)})   ], \phi \circ \nu = (-1)^{k-1} \phi( \alpha(x^{(1)}) \wedge \nu(x^{(2)})), 
	\end{equation}
	where we used the Sweedler notation $\Delta(x) = x^{(1)} \otimes x^{(2)}$ for the coproduct in $\Lambda^\bullet \g$. This shows that $\delmm \phi =(-1)^{k-1} \bb{\nu,\phi}$. 

	The second statement follows from 
	\begin{align*}
		\delma \chi_1 \circ \chi_2 + \chi_1  \delma \circ \chi_2 &=  \alpha' \chi_1  \circ \chi_2  -   \chi_1 \alpha'\circ \chi_2  + \chi_1 \circ \alpha' \chi_2 - \chi_1 \chi_2 \alpha' \\ &= 
		\alpha' \chi_1  \circ \chi_2  - \chi_1 \chi_2 \alpha' = \delma (\chi_1 \circ \chi_2).
	\end{align*}
	Using $ \chi_1 \circ\alpha' \chi_2 = \mu (\chi_1 \otimes \alpha^\chi_1) Q (\alpha^{\chi_2} \alpha' \otimes \alpha' \chi_2) =  \mu (\chi_1 \alpha'\otimes \alpha^\chi_1 \alpha') Q (\alpha^{\chi_2}  \otimes \chi_2) = \chi_1 \alpha'\circ \chi_2 $. Also remember that $\alpha' \chi_1$ and $\delma \chi_1$ have the height $\abs{\chi_1} + \abs{\alpha}$.
\end{proof}

\section{\texorpdfstring{$\alpha$}{alpha}-type cohomology for Hom-Lie bialgebras}\label{sc:alphabicoho}

We  define an $\alpha$-type  cohomology for Hom-Lie bialgebras, where $\alpha =\beta$. We only give the definition, state the basic facts  and do not give any details of the proofs here. 
The case $\alpha$ different from  $\beta$ is more complicated than the case of Hom-associative bialgebras, since the compatibility of the product and the coproduct involves the structure maps $\alpha$ and $\beta$. 
Notice that if  $\alpha=\beta$,  the complex is no longer a bicomplex.

 For a Hom-Lie bialgebra $(\g,\nu,\delta,\alpha,\alpha)$ we set
\begin{equation}
	\abLC^k (\g)=  \bLC_\nu^k(\g) \oplus \bLC_\alpha^k(\g) = \bigoplus_{l=1}^{k} \Hom(\Lambda^l \g, \Lambda^{k-l+1}\g) \oplus \bigoplus_{i=1}^{k-1} \Hom(\Lambda^l \g, \Lambda^{k-l} \g). 
\end{equation}
Fur further set $\bLC^{ij}_\alpha(\g) = \Hom(\Lambda^i \g, \Lambda^{j} \g) \subset \bLC_\alpha^{i+j}(\g)$ and  $\bLC^{ij}_\nu(\g) = \Hom(\Lambda^i \g, \Lambda^{j} \g) \subset \bLC_\nu^{i+j-1}(\g)$.

Similarly to the previous section it is convenient to define the height of an element in the cohomology. Since here we do not distinguish between $\alpha$ and $\beta$, the height is an integer and not a pair. With this we define a grand crochet  as before.   
The height of an element in $\bLC^{ij}_\alpha(\g)$ is $i+j-1$ and of one in $\bLC^{ij}_\nu(\g)$ is $i+j-2$.

For $\phi_{ij} \in \Hom(\Lambda^i \g, \Lambda^{j} \g) \subset \bLC_\nu^{i+j-1}(\g)$ and $\psi_{ij} \in \Hom(\Lambda^i \g, \Lambda^{j} \g) \subset \bLC_\alpha^{i+j}(\g)$, we set 
\begin{align}
	\del \phi &= (\delmm \phi , \delma \phi , \delmm^c \phi ) \in \bLC^{i+1,j}_\nu \oplus \bLC^{i,j}_\alpha \oplus \bLC^{i,j+1}_\nu \\
	\begin{split}
	\del \psi &= ( \delam \psi ,-\delaa \psi, \del_b \psi,-\delaa^c \psi , \delam^c \psi ) \\
	&\in \bLC^{i+2,j}_\nu \oplus \bLC^{i+1,j}_\alpha  \oplus \bLC^{i+1,j+1}_\nu  \oplus \bLC^{i,j+1}_\alpha \oplus \bLC^{i,j+2}_\nu. 
	\end{split}
\end{align}
Here $\delaa  = \bb{\nu, \cdot}$, $\delaa^c = \bb{\delta, \cdot}$, $\delma \phi = \alpha \phi - \phi \alpha$ and $\delma \phi = \mu ( \alpha^{\abs{\phi-1}} \nu  \wedge \phi)\Delta$. 
So there are the same as in the definition of the $\alpha$-type cohomology for Lie algebras and their duals. 
The map $\del_b$ is defined by 
\begin{align}
	\begin{split}
		\del_b \psi &= \del_{b1} + \del_{b2} \\
 		&=\mu^3(\id \otimes \nu \otimes  \id)(\pr \otimes \mu \otimes \id )(\Delta \otimes \Delta) ( \delta \alpha^{ \abs{\psi}-1} \otimes \psi )\Delta  \\
		&\quad +\mu (\alpha^{ \abs{\psi}-1} \nu  \otimes \psi )(\mu \otimes \mu) (\pr \otimes \Delta \otimes \id) (\id \otimes \delta \otimes  \id)\Delta^3.
	\end{split}
\end{align}
This means for $\psi \in \bLC^{k,l}_\alpha(\g)$
\begin{align*}
	\del_{b1} \psi (x_1,\dots,x_{k+1}) &= \sum_{i=1}^{k+1} (-1)^i \delta(\alpha^{k+l} x_i) \cdot \psi(x_1, \dots ) \\
	&=   \sum_{i<j} (-1)^{i+j-1} \alpha^{k+l}(x_i^{[1]}) \wedge  (\alpha^{k+l}(x_i^{[2]}) \cdot  \psi(x_1, \dots,\hat x_i,\dots, x_{k+1} )),
\end{align*}
where $\delta(x) = x^{[1]}\wedge x^{[2]}$, and $\del_{b2}$ is the dual given by
\begin{align*}
	\del_{b2} \psi (x_1,\dots,x_{k+1}) &= \sum_{i=1}^{k+1} (-1)^i [\alpha(x_i),x_j^{[1]}] \wedge \psi(x_j^{[2]},\alpha(x_1),\dots,\hat x_i,\hat x_j, \alpha(x_{k+1})).
\end{align*}

\begin{theorem}
	The map $\del$ defined above  is a differential for the complex $\abLC(\g)$.
\end{theorem}
\begin{proof}
 This is a very lengthy calculation, which we avoid here.
\end{proof}

\begin{remark}
	The differential is defined such that the first order term of a  deformation is a cocycle. 
	Furthermore similar to \cref{sc:linfty} we found the low order terms of an $L_\infty$-structure which has as Maurer-Cartan elements Hom-Lie algebras and the cohomology in low degrees  one gets from this is the one describe above. So the statements relating cohomology and deformations are true. 		
\end{remark}

Given a Lie bialgebra $\g$  and an endomorphism $\gamma:\g \to \g$, we define a cohomology for $\gamma$. For the case of arbitrary Lie bialgebra morphisms this can been found in \citep{fregier2}. We modify the definition slightly to better agree with the definition of $\abLC(\g)$. 
The complex is given as the complex for the $\alpha$-type cohomology. We write however $C_\nu$ instead of $\bLC_\nu$ and $C_\gamma$ instead of $\bLC_\alpha$.
The differential for $\phi \in C_\nu^{i,j}(\gamma)$ and $\psi \in C_\gamma^{i,j}(\gamma)$ is given by 
\begin{align}
	\del \phi &=  (\del_{CE} \phi , \del_\gamma \phi, \del^c_{CE} \phi) \in C^{i+1,j}_\nu \oplus  C^{i,j}_\alpha \oplus  C^{i,j+1}_\nu , \\
	\del \psi &=  (-\del_{CE} \psi , -\del^c_{CE} \phi) \in C^{i+1,j}_\alpha  \oplus  C^{i,j+1}_\alpha . 
\end{align}
Here $\del_{CE}$ is the ordinary Chevalley-Eilenberg differential and $\del_{CE}^c$ its dual. The action on $\psi$ is given by $x \cdot \psi = \gamma(x) \psi$.

There is a morphism $\Phi: C^\bullet(\gamma) \to \abLC^\bullet(\g_\gamma)$ of complexes given by 
\begin{align}
	\Phi(\phi) &= ( \gamma^{i-1} \phi \gamma^{j-1}), \\
	\Phi(\psi) &= ( \gamma^{i-1} ( \psi  \circ \nu) \gamma^{j-1} , \gamma^{i-1} \psi \gamma^{j-1},  \gamma^{i-1} ( \delta  \circ \psi) \gamma^{j-1} ),
\end{align}
for $\phi \in C_\nu^{i,j}$ and $\psi \in C_\gamma^{i,j}$.
If $\gamma$ is invertible, so is $\Phi$.

\section{Hom-type cohomology for Hom-Lie bialgebras}\label{sc:bialgcoho}

In this section we provide a Hom-type cohomology for Hom-Lie bialgebras $(\g,\nu,\delta,\alpha,\beta)$, where $\alpha$ and $\beta$ are arbitrary.

We consider the complex $$B^{i,j} = \Hom_\alpha(\Lambda^i \g,\Lambda^j \g) = \{ \phi \in \Hom(\Lambda^i \g,\Lambda^j \g) | \alpha^{\otimes j} \phi = \phi \alpha^{\otimes i} , \beta^{\otimes j} \phi = \phi \beta^{\otimes i}  \}$$ for $i,j \geq 1$ and $\{0\}$ otherwise.

\begin{theorem}
	The bracket defined in \cref{eq:bb} defines a graded Lie algebra structure on $B^{\bullet\bullet}$.
\end{theorem}
\begin{proof}
	We compute $\phi \circ( \psi \circ \chi)$, for this we introduce $$Q = (\mu \otimes \id)(\id \otimes \pr \otimes \id) (\id\otimes \Delta),$$ and 
	get 
	\begin{align}
		\phi & \circ( \psi \circ \chi) = \mu ( \phi \otimes \alpha^\phi)(\id \otimes \mu) Q (\id \otimes \psi \otimes \alpha^\psi) (\id \otimes Q) ( \alpha^{\phi\psi} \otimes \alpha^\chi \otimes \chi) (\id \otimes \Delta) \Delta \nonumber                      \\
		     & = \mu (\phi \otimes \alpha^\phi) ( \id \otimes \mu) (Q \otimes \id) ( \alpha^\psi \otimes \psi \otimes \alpha^\psi)                                                                                            
		(\id \otimes Q) ( \Delta \otimes \id) (\alpha^\chi \otimes \chi) \Delta  \label{eq:bb1} \\
		     & + \mu^3 ( \alpha^\phi \otimes \phi \otimes \alpha^\phi) ( \id \otimes Q) ( \tau \otimes \id) ( \alpha^\psi \otimes \psi \otimes \alpha^\psi) ( \alpha^{\chi} \otimes \alpha^\chi \otimes \chi) \Delta^3  \nonumber .
	\end{align}
	Here we use the abbreviations $\mu^3 = \mu(\id \otimes \mu)$ and $\Delta^3 = (\Delta \otimes \id) \Delta$.
	Using that $\phi$ commutes with $\alpha$ and $\beta$, the last term can be rearranged to 
\begin{align}
		\mu^3 ( \alpha^\phi \otimes \alpha^\psi \otimes \id)(\psi \otimes \phi \otimes \id) ( \id \otimes Q) (\tau \otimes \id)( \id \otimes Q)(\alpha^\chi \otimes \alpha^\chi \otimes \chi) \Delta^3. 
\end{align}
	This can be seen to be symmetric in $\phi$ and $\psi$.
	The term $(\phi \circ \psi) \circ \chi$ can be computed similarly. One gets the term~\eqref{eq:bb1} plus a term symmetric in $\psi$ and $\chi$. 
	So in total we get $\sum_{perm. \phi, \psi,\chi} ( \phi \circ( \psi \circ \chi) - (\phi \circ \psi) \circ \chi) =0$. This is equivalent to the fact that $\bb{\cdot,\cdot}$ satisfies the Jacobi identity.
\end{proof}

\begin{prop}
	Let $\g$ be a vector space with two structure maps $\alpha,\beta :\g \to \g$ then 
	\begin{itemize}
		\item If $\nu \in B^{2,1}$ satisfies $\bb{\nu,\nu}=0$  then $(\g,\nu,\alpha)$ is a  Hom-Lie algebra. 
		\item If $\delta \in B^{2,1}$ satisfies $\bb{\delta,\delta}=0$ then $(\g,\delta,\beta)$ is a  Hom-Lie coalgebra. 
		\item If a  pair $(\nu,\delta)$ satisfies $\bb{\nu +\delta,\nu +\delta}=0$ then $(\g,\nu,\delta,\alpha,\beta)$ is a  Hom-Lie bialgebra. These are the Maurer-Cartan elements of $\bb{\cdot,\cdot}$.		
	\end{itemize}
\end{prop}
\begin{proof}
	This follows easily using the computation done in  the proof of \Cref{th:bbdif}.
\end{proof}

Moreover the grand crochet can be restricted to the complex $\Hom_\alpha(\Lambda^\bullet\g,\g)$, which is the complex of the Chevalley-Eilenberg cohomology for Hom-Lie algebras. On this complex, it  agrees with the generalization of the Nijenhuis-Richardson bracket given in \citep{homcoho}.

We define the total complex $B^i = \bigoplus_{j=1}^i B^{j,i-j}$.
\begin{prop}
	Let $\g$ be a Hom-Lie bialgebra.
	The map $\del: B^\bullet \to B^\bullet$ defined by $\del \phi = \bb{\nu+ \delta, \phi}$ is a differential, i.e. $\del\circ\del =0$.
\end{prop}
\begin{proof}
	This is standard and follows from the Jacobi identity of the grand crochet.
\end{proof}

With this one can define a cohomology for Hom-Lie bialgebras.
 It can be used to describe deformations of Hom-Lie algebra, coalgebra and bialgebras, where  the structure maps are fixed.

\begin{remark}
	Note that in the Hom-context it is not straightforward to define an analog to quasi-Lie bialgebras. Since extending the bracket to including maps in $\Hom(\K,\Lambda^\bullet \g)$ would involve $\alpha^{-1}$.
\end{remark}

\begin{remark}
	In the case that $\alpha= \beta = \id$ the cohomology we defined here is precisely the normal cohomology for Lie bialgebras, since in this case the complex is just $\Hom(\Lambda^\bullet \g,\Lambda^\bullet g)$ and the grand crochet we defined here is the ordinary one.
\end{remark}

\bibliographystyle{bibstyle}
\bibliography{bibli}

\end{document}